\documentclass[14]{amsart} 
   \usepackage{hyperref}      
\usepackage{amsmath, amssymb, mathrsfs, amsfonts,  amsthm} 
\usepackage [all]{xy}

\newtheorem {theorem} {Theorem} [subsection] 
 
\newtheorem {conjecture} [theorem] {Conjecture}
\newtheorem{lemma}[theorem] {Lemma}

\newtheorem {definition} [theorem] {Definition} 
\newtheorem {proposition}  [theorem]{Proposition} 

\newtheorem {corollary}[theorem]  {Corollary} 
\newtheorem {example} [theorem]  {Example}
  
\newtheorem {notation} {Notation} [section]
\newtheorem {remark} [theorem] {Remark}
\numberwithin {equation} {section}
\DeclareMathOperator{\image}{image}
 
\begin{document} 
\author {Yasha Savelyev}
\address{savelyev@math.umass.edu, Department of Mathematics
 and Statistics,  Lederle Graduate Tower, University of Massachusets, 
 Amherst, MA, 01003} 
\title   {Spectral geometry of  the group of Hamiltonian symplectomorphisms}
\begin{abstract} We introduce here a  natural functional associated to any $b
\in QH_* (M, \omega)$: \emph{spectral length functional}, on the space of
``generalized paths'' in $ \text {Ham}(M, \omega)$, closely related to both the Hofer length
functional and spectral invariants and establish some of its properties. 
This
functional is smooth on its domain of definition, and moreover the nature of
extremals of this functional suggests that it may be variationally complete,  in
the sense that any  suitably generic element of $ \widetilde{\text {Ham}}(M,
\omega)$ is connected to $id$  by a generalized path minimizing spectral length. 
Rather strong evidence is given for this when $M=S ^{2}$, where we show that all the
Lalonde-McDuff Hamiltonian symplectomorphisms are joined to id by such a path. We
also prove that the associated norm on $ {\text {Ham}}(M,
\omega)$ is non-degenerate and bounded from below by the   the spectral norm.
If the spectral length functional is variationally complete the associated norm
reduces to the spectral norm.
 \end{abstract}
\maketitle 
\section {Introduction} 
The Hofer length functional and the resulting Finsler geometry on $ \text
{Ham}(M, \omega)$: the Hofer geometry, has been one of the driving forces in
symplectic geometry. It provides a direct link from the world of  symplectic
topology to the world of metric spaces. This is combined with  its naturality, 
simplicity of definition, and intriguing connections with ``hard'' symplectic
geometry: for example in the sense of elliptic methods of Gromov and Floer. 

However this functional has  a disturbing variation incompleteness problem, as
Lalonde-McDuff \cite{Energy.stab.I} observed that in $ \text {Ham}(S ^{2},
 \omega)$, there are elements not joined by locally (in the path space) length
 minimizing Hofer geodesics, i.e. stable geodesics. For future reference we
 shall refer to these elements by \textbf{Lalonde-McDuff symplectomorphisms}. 
 This is one of the  problems stalling progress in understanding global
 metric properties of $ \text {Ham}(M, \omega)$. 
 
 The main goal here is to describe another
highly natural functional $L ^{s} _{j}$, depending on an almost complex
structure $j$ on $M$, on a certain ``generalized path space'' of $ \text
{Ham}(M, \omega)$.  This functional is no longer so elementary 
in nature and is deeply intertwined
 with chain level Floer theory, and spectral invariants. It may seem awkward
 that $L ^{s} _{j}$ depends on $j$, but  surprisingly extremals for $L
 ^{s} _{j}$ and their lengths are independent of $j$. Consequently if $L ^{s}
 _{j}$ was suitably variationally complete the
associated distance function would depend only on $\omega$, and in this case
we show that the associated spectral distance from $id$ to $ \widetilde{\phi}
\in \widetilde{\text {Ham}}(M, \omega)$ reduces to the spectral norm of $
\widetilde{\phi}$. 

Here is the main criterion for being $L ^{s} _{j}$
minimizing that is used in this paper.
 \begin{proposition} \label{thm.adiabatic} Suppose
a pair of regular $ \widetilde{\phi} _{\pm} \in \widetilde{\text {Ham}}(M,
\omega)$ can be connected by a strong Cerf homotopy in degree 2n, then there is
an $L ^{s} _{j}$ minimizing generalized path between $ \widetilde{\phi} _{\pm}$.
\end{proposition}
A strong Cerf homotopy $H _{r}: M \times S ^{1} \to \mathbb{R}$, $0 \leq
r \leq 1$, in degree 2n,  is roughly speaking a Cerf-Floer homotopy with 
nice behavior in CZ degree 2n. 
Details are given in Section \ref{section.exotic}.  To state our main results we
need to give the full background.
\subsection{Spectral length functional}
\subsubsection {Hamiltonian category $ \mathcal {G} _{M}$}
\label{section.hamcategory} Let $(M, \omega)$ be a closed symplectic manifold. To this we associate a
topological category  whose objects are lifts to the universal cover
of Hamiltonian symplectomorphisms of $M$ and the ``simple'' morphism space $
\mathcal {P} ^{s} _{ {\pm}}$ from $ \widetilde{\phi} _{-}$ to
$ \widetilde{\phi} _{+}$ consists of Hamiltonian connections $ \mathcal {A}$ on
the topologically trivial bundle $$\pi: P = M \times ( \mathbb{R} \times S ^{1}) \to \mathbb{R} \times S ^{1},$$ with the property
that  $\mathcal{A}$ over $r \times S ^{1} \subset \mathbb{R} \times S ^{1}$
stabilizes to $\mathcal {A} _{-}$, respectively $ \mathcal {A} _{+}$ for
sufficiently small, respectively large $r$, and the holonomy of $ \mathcal {A}
_{\pm}$ is $ \widetilde{\phi} _{\pm} \in  \widetilde{\text {Ham}}(M, \omega)$.
Also for $r$ sufficiently small, respectively large  $ \mathcal {A}$ is assumed to be flat in the $r$
direction, i.e. the horizontal lift of $ \frac{\partial}{\partial r}| _{z}$ 
 is $\frac{\partial}{\partial r}| _{x, z}$, (and consequently $ \mathcal {A}$ is
 flat outside a compact region of $P$ by above assumptions). The compact region of
 $P$, where $ \mathcal {A}$ is not assumed to be $r$-flat, will be called \emph{principal}. 
 \begin{remark} We should be able to relax the notion of morphism to allow $ \mathcal {A}$, 
 which are only asymptotic to flat connections as $r
\hookrightarrow \pm \infty$, converging sufficiently fast. All of the discussion
in this paper should go through for this case. We have no use for such morphisms
here, but they may potentially be useful when studying the question of
variational completeness more closely. We should also be able to take 
higher genus Riemann surfaces with boundary instead of $
\mathbb{R} \times S ^{1}$ and take tuples $ \{ \widetilde{\phi} _{i}\}$ of lifts
of Hamiltonian symplectomorphisms as objects, analogously to
\cite{LalondeFieldtheory}.
\end{remark}

We then partially compactify this space by allowing SFT style degenerations of
the morphisms. The full morphism space $ \mathcal {P} ( \widetilde{\phi} _{-},
\widetilde{\phi} _{+})$ often abbreviated $\mathcal {P} _{\pm}$, then consists
of simple morphisms and formal concatenations of simple morphisms. The identity morphism $id _{ \widetilde{\phi}}$ in this category is just an $r$-flat connection $ \mathcal
{A} _{ \widetilde{\phi}}$, with holonomy over $r \times S ^{1}$ being $
\widetilde{\phi}$ for every $r$. For this to be a true category we need to put a mild equivalence relation, which we
suppress. (More appropriately this is a 2-category.) 
 \subsubsection {From paths in $ \text {Ham}(M, \omega)$ to morphisms of $
\mathcal {G} _{M}$} \label{sec.path.groupoid} Let
 $$p: [0,1] \to {\text {Ham}}(M, \omega),$$ $p (0)= id$, be generated by $H: M
 \times [0,1] \to \mathbb{R}$. 
To such a $p$ there is a
 naturally associated Hamiltonian connection $\mathcal{A} _{p}$ 
 on  $P \to \mathbb{R} \times S ^{1}$, which 
is a morphism from $id$ to $p (1)$ in $ \mathcal {G} _{M}$. Explicitly $ \mathcal
{A} _{p}$ is constructed as follows: for $ 0 \leq r \leq 1$ the
 horizontal lift of $ \frac{\partial}{\partial \theta} \in T _{r, \theta}
 (\mathbb{R} \times S ^{1})$ is determined by the condition that the holonomy
 path $\phi _{\theta}$, $\theta \in S ^{1}$, of $ \mathcal {A}| _{r \times S
 ^{1}}$ is generated by $ \eta (r) \cdot H$, while being flat in the $r$
 direction, where 
  $\eta: \mathbb{R} \to \mathbb{R}$ is a fixed function
 satisfying:
 \begin{equation} \label {eq.eta} \eta (r) = \begin{cases} 1 & \text{if
  } 1 -\delta \leq r \leq 1 ,\\ r   & \text{if } 2 \delta \leq r \leq 1-2\delta,
  \\ 0 & \text {if } r \leq \delta
   \end{cases}
\end{equation}
 for a small $\delta >0$. 
\subsubsection {Floer chain complex} \label{section.Floer.chain.complex} 
For more details on Floer homology and Hamiltonian perturbations
 the reader may see for example \cite[Chapter
8, 12]{MS}. This is a chain complex associated to a 
 $ \widetilde{\phi} \in \widetilde{\text {Ham}}(M, \omega)$. We fix an
 $\omega$-compatible complex structure $j$ on $M$ once and for all. For the definition of the chain
complex we need a suitably generic Hamiltonian connection $ \mathcal {A}
_{ \widetilde{\phi}}$  on $M \times S ^{1}$ with holonomy $ \widetilde{\phi}$,
such a connection will be called Morse-Floer. Then $CF_* ( \mathcal {A} _{
\widetilde{\phi}})$ is generated over $ \mathbb{Q}$ by pairs $
\widetilde{\gamma}=(\gamma, \bar{\gamma})$, with $\gamma$ a flat section of $
\mathcal {A} _{ \widetilde{\phi}}$ and $ \bar{\gamma}$ an equivalence class of 
a bounding section in $M \times D ^{2}$, with two bounding disks considered equivalent if their difference (as 2-chains) is annihilated by $c_1 (TM)$. The grading is
given by Conley-Zehnder index. The boundary operator is defined through count of
$J _{ \mathcal {A} _{ \widetilde{\phi}}}$-holomorphic sections of $P$, whose
projections to $M \times S ^{1}$ are asymptotic in backward,  forward time to
generators of $CF_* ( \mathcal {A}_{ \widetilde{\phi}})$. Here $J _{ \mathcal
{A} _{ \widetilde{\phi}}}$ is the complex structure induced by the extension of
the connection $ \mathcal {A} _{ \widetilde{\phi}}$ to a connection on $P$, flat
in the $r$ direction, more details are given below.
\subsubsection{From morphisms of $ \mathcal {G}$  to almost complex structures
on $P$}
Given $ \mathcal {A} \in \mathcal {P} _{\pm}$  there is a
\emph{$\pi$-compatible} almost complex structure $J _{\mathcal{A}}$ on
 $P$ induced by $\mathcal{A}$, which means the following:
\begin{itemize} 
  \item The natural map
$\pi: (P, J _{ \mathcal {A}})
\to ( \mathbb{R} \times S ^{1}, j)$ is holomorphic.  
\item  $J _{ \mathcal {A}}$ preserves the horizontal subbundle $Hor
^{ \mathcal {A}}$ of $TX$ induced by $ \mathcal {A}$.
\item $J _{ \mathcal {A}}$ preserves the vertical tangent bundle of $M
\hookrightarrow P \to \mathbb{R} \times S ^{1}$, and restricts to the fixed
complex structure $j$ on $M$.
\end{itemize}
\begin{notation} To cut down on notation we will sometimes forgo notationally 
distinguishing between $\phi \in \text {Ham}(M, \omega)$ and its lifts to
universal cover.
\end{notation}
\subsubsection {Definition of the spectral length functional} Let $(M, \omega)$
be a closed monotone symplectic manifold with a positive monotonicity constant and 
let $ \mathcal {A} _{\pm}=\mathcal {A} _{\phi _{\pm}}$ be Morse-Floer.
Let $ \{\widetilde{\gamma}  ^{-}_{i} \}$,  $ \{\widetilde{\gamma}  ^{+}_{j} \}$
be a collection of natural geometric generators for $CF _{k}( \mathcal
{A}_{-})$, respectively $CF _{k}(
\mathcal {A}_{+})$  over $ \mathbb{Z}$, and let $\mathcal {M} ( \widetilde{\gamma}
^{-}_{i},  \widetilde {\gamma} ^{+}_{j})$ 
denote the moduli space (whose virtual dimension is 0 by the index theorem) of
holomorphic sections of $P$ asymptotic to $ \widetilde{\gamma} _{i} ^{-}$, $
\widetilde{\gamma} _{j} ^{+}$.


 
 Now let $K _{\mathcal A}: \mathbb{R} \times S ^{1} \to C
 ^{\infty} (M)$ denote the Hodge star of the Lie algebra valued curvature form
 $R _{
 \mathcal {A}}$ of $\mathcal {A}$, with respect to standard Kahler metric $g_j$ on $ \mathbb{ \mathbb{R}}
  \times S ^{1}$. In other words $K _{\mathcal A} (z) \in C ^{\infty} (M)$ is the
  Lie algebra element $R _{ \mathcal {A}} (v, w)$, for $v, w$ an orthonormal pair at $z \in  \mathbb{R} \times S ^{1}$. 
  However, we will instead think of $K
   _{\mathcal {A}}$ in terms of  the naturally associated function $K _{\mathcal
   {A}}: P \to \mathbb{R}$. 
We now define a variant of the continuation map: $$ $$ 
\begin{align} \label {eq.contE} \Psi _{E} ^{k}: CF_k ( \mathcal {A}
_{-}) \to CF_k(\mathcal {A}_+), \\
\Psi _{E} ^{k}( \widetilde{\gamma} _{i} ^{-}) = \sum _{j} \# \mathcal {M} (
\widetilde{\gamma} ^{-}_{i},  \widetilde {\gamma} ^{+}_{j}) _{E}
\,\widetilde{\gamma} ^{+} _{j},
\end{align}
where $\mathcal {M} ( \widetilde{\gamma} 
^{-}_{i},  \widetilde {\gamma} ^{+}_{j}) _{E}$ is the space of $J _{
\mathcal {A}}$-holomorphic sections $u$ which satisfy \begin{equation} \label
{eq.integral}\int _{u} K _{ \mathcal {A}} \,\pi ^{*} \omega _{st} \leq E,
\end{equation}
with $\omega _{st}$ denoting the standard area form on $ \mathbb{R} \times S
^{1}$. 
 \begin{definition}  
For a degree $k$ element $b \in QH_* (M)$ we define
\begin{equation*} L^s _{j} (\mathcal{A}, b) =  \inf \{E| \text {s.t. } PSS(b) \in
FH_* ( \mathcal {A} _{+})
\text { can be represented by a chain in }\image \Psi _{E} ^{k} \}.
\end{equation*} 

\end{definition} 

 Note that by our assumptions on $ \mathcal {A}$, $K _{\mathcal {A}}$
vanishes outside a compact subset of $P$ so that the integral
\eqref{eq.integral} is finite. When $b$ is the fundamental class we 
abbreviate $L ^{s} _{j} ( \mathcal {A}, [M])$ by $L ^+ _{j} ( \mathcal {A})$.
When $ \widetilde{\phi} _{-}= id$, we define $L ^{+} _{j} ( \mathcal {A})$
naturally, by dropping constraints on the left in the definition of $\Psi ^{k}
_{E}$, i.e. we just count elements in $ \mathcal {M} ( \widetilde{\gamma}  ^{+}
_{j})$. This will be an important special case.

Our assumption that $M, \omega$ is monotone with positive monotonicity constant
insures that only finitely many holomorphic sections $u$ come into the
definition of $\Psi _{E} ^{k}$, and consequently $L ^{s} _{j}$ is a priori defined
and is \emph{smooth} on an open dense subspace of $ \mathcal {P} _{{\pm}}$, 
consisting of \emph{regular} $ \mathcal {A}$, i.e. those for which the associated 
CR operator is surjective for each $u \in \mathcal {M} (
\widetilde{\gamma} ^{-}_{i},  \widetilde {\gamma} ^{+}_{j})$. 

 Let $ \overline {
\mathcal {A}} =(id \times \sigma) ^{*}\mathcal {A}$, where 
\begin{equation} \label {eq.overline} id \times \sigma:
M \times ( \mathbb{R} \times S ^{1}) \to M \times ( \mathbb{R} \times S ^{1}),
\end{equation}
 and where
$\sigma$ is an orientation preserving reflection. Then $\overline {\mathcal
{A}}$ is a morphism from $\phi_+ ^{-1}$ to $\phi_- ^{-1}$. We set  $L ^{-} _{j}
( \mathcal {A}) = L ^{s} _{j} ( \overline{ \mathcal {A}}, [M])$. 
Define $$ L ^{s} _{j} ( \mathcal {A})= L ^+ _{j} ( \mathcal {A}) - L ^{-} _{j} (
\mathcal {A}).$$
and define a function $\widetilde{\text {Ham}}(M, \omega) \to \mathbb{R}$: 
\begin{align*}  \widetilde{\phi} \mapsto | \widetilde{\phi}| _{j} \equiv
\inf _{ \mathcal {A} \in \mathcal
 {P} (id, \widetilde{\phi})} |L ^{s} _{j} ( \mathcal {A})| ,
\end{align*}
We will use the
same notation for the pushdown of this function to $ \text {Ham}(M, \omega) $, i.e. 
\begin{equation*} |\phi| _{j}= \inf \{ | \widetilde{\phi}| _{j} \text { s.t. }\,
\widetilde{\phi}\text { projects to }  \phi \}.
\end{equation*}
\subsection {Main results}
\begin{proposition} \label{theorem.main} The  ``norm" $|\cdot|_j$ is
non-degenerate.
Furthermore \begin{equation*} L ^{s} _{j} ( \mathcal {A}) \geq 
|\widetilde{\phi}_+|_{s} - | \widetilde{\phi} _{-}| _{s},
\end{equation*}
where $ | \widetilde{\phi}| _{s}$ is the usual spectral pseudo norm of
Oh-Schwarz on the universal cover.  
\end{proposition}
\begin{definition} We call $ \mathcal {A}$ quasi-flat if  $L ^{s} _{j} (
\mathcal {A})=| \widetilde{\phi}_+|_{s} - | \widetilde{\phi} _{-}| _{s}$. 
\end{definition}
By above proposition, such $ \mathcal {A}$ are $L ^{s} _{j}$ minimizing on $
\mathcal {P} _{\pm}$. The name is due to fact that such $ \mathcal {A}$ have
certain distinguished flat sections, which will be apparent from the definition
of spectral length to be given later. This condition could also be viewed as
relaxation of the condition for being $L ^{+}$ minimizing, which translated into language of connections 
requires that $ \mathcal {A}$ has  a special,
constant flat section. 

Inspired by this we conjecture: \begin{conjecture} \label{hypothesis.1} For  $ \widetilde{\phi} \in \widetilde{\text {Ham}}(M, \omega)$
 Floer non-degenerate,  $L ^{s} _{j}$  attains a minimum on the generalized path
 space $ \mathcal {P} ({id, \widetilde{\phi}})$. Moreover, this minimum can be
 represented by a quasi-flat morphism, and so 
\begin{equation*} | \widetilde{\phi}| _{j}= | \widetilde{\phi}| _{s}.  
\end{equation*}
\end{conjecture}
 As a step towards this conjecture we have: 
 \begin{theorem} \label{thm.S2} The class of
 Lalonde-McDuff Hamiltonian symplectomorphisms of $S ^{2}$, can be joined to $id$ by quasi-flat and hence $L ^{s} _{j}$ minimizing
 morphisms.
\end{theorem}

Consider the direct analogue of the Hofer length functional defined on $
\mathcal {P} _{\pm}$ by 
 \begin{equation*}
 L _{H} ( \mathcal {A}) = \int _{ \mathbb{R} \times S ^{1}} \max _{M} K _{ \mathcal
 {A}} (z)-\min _{M} K _{ \mathcal {A}} (z) \, \omega _{st}. 
 \end{equation*}
Here is another corollary:
\begin{corollary} \label{corollary.main}The distance function $d _{H}$ induced
by $L _{H}$ gives a non-degenerate norm on $ {\text {Ham}}(M,
\omega)$, bounded from below by the spectral norm.
\end{corollary}
\begin{proof} We clearly have $L _{H} ( \mathcal {A}) \geq L ^{s} _{j} ( \mathcal {A})$
and so the corollary follows.
\end{proof} 
\subsection {Applications to Hofer geometry}
It may be difficult to see how the above theory can be useful in classical Hofer
geometry. In fact the applications we have in mind require another conjecture,
which we now give.  First a definition:
\begin{definition} Define the \emph { \textbf{pseudo-injectivity}} radius of a
symplectic manifold $M, \omega$ as 
\begin{align*} inj (M,\omega) = \sup \{r | \text { s.t. }|\widetilde{\phi}|_j
^{s} < r \Rightarrow  \text { space of quasi-flat morhispms in } \\\mathcal {P} (id,
\widetilde{\phi}) \text { is non-empty and contractible}.\}
\end{align*}
Note that this quantity is $j$ independent.
\end{definition}
\begin{conjecture} 
\begin{equation*} inj (S ^{2}, \omega _{st}) =1, \text { where} \int _{S ^{2}}
\omega _{st}=1.
\end {equation*}
 Moreover the space fibering over the Hofer $(1-
 \epsilon)$-ball in $ \text {Ham}(S^2, \omega _{st})$, $B (1-\epsilon)$ with
 fiber over $\phi \in B (1 -\epsilon)$: the space of quasi-flat morphisms from
 $id$ to $\phi$, is a Serre fibration. (For any $\epsilon>0$.)
\end{conjecture}
It is well known that the Hofer $ (1-\epsilon)$-ball in $ \text {Ham}(S ^{2},
\omega _{st})$ is simply connected, which is why in the above formulation we did
not mention the universal cover. Given this  we may immediately deduce
contractability in $C ^{\infty}$ topology of  Hofer $(1 -\epsilon)$-ball in $
\text {Ham}(S ^{2}, \omega)$, which by itself is an open question and its solution opens the door on other
interesting open problems. 
\subsection{Acknowledgements} Special thanks to Yong-Geun Oh for patiently
listening to very preliminary ideas and making some excellent
suggestions, Dusa McDuff and Leonid Polterovich for discussions, as well as the
anonymous referee for helpful criticism. 

\section {Properties of the spectral length functional}
\label{section.properties}
\subsection {Some notations and conventions}
The  action
functional $$A _{H}: \widetilde{\mathcal {L} M} \to \mathbb{R},$$ 
is defined by
$$A _{H} (\gamma, D) = - \langle \omega, D \rangle + \int _{0} ^{1} H (\gamma (t), t) dt,$$
where $H: M \times S ^{1} \to \mathbb{R}$. The induced 
 Hamiltonian flow $X _{t}$ is given by
\begin{equation*} \omega (X _{t} , \cdot) = -dH _{t} (\cdot). 
\end{equation*}
The positive Hofer length functional $L ^{+} (p)$, for a path $p$ in $ \text
{Ham}(M, \omega)$ is defined by
\begin{equation*} L ^{+} (p) = \int _{0} ^{1} \max H  _{t} \; dt,
\end{equation*}
where $H _{t}$ is the generating Hamiltonian for $p$, normalized by
condition 
\begin{equation} \int _{M} H  _{t} \cdot \omega ^{n} =0.
\end{equation}
While the Hofer length functional $L$ is defined by 
\begin{equation*} L (p)= \int _{0} ^{1} \max H _{t} - \min H _{t}\; dt.
\end{equation*} 
 The Conley-Zehnder index is normalized so that
for a $C ^{2}$-small Morse function $H_t = H$,  the orbits corresponding to
critical points (with trivial bounding disks) have CZ equal to Morse index. For
 future reference we note that under above condition,  the CZ index of a
 constant orbit $(x, D)$ for $x$ a critical point and $D$ an equivalence class of a bounding
disk is the Morse index of $x$ minus $2 \langle c_1 (TM), D \rangle$.

\begin{proof} [Proof of Proposition \ref{theorem.main}]
From  $ \mathcal {A}$ we
may construct a certain remarkable closed 2-form  $\Omega _{ \mathcal {A}}$,
called the coupling form, originally  constructed in \cite{GuileminSternberg}.  This form
is characterized as follows, the restriction of $\Omega _{ \mathcal {A}}$ to fibers $M$ of $P \to \mathbb{C}$ coincides
with $\omega$, the $\Omega _{ \mathcal {A}}$-orthogonal subspaces in $TP$ 
to fibers $M$ are the horizontal subspaces 
and the value of  $\Omega _{ \mathcal {A}}$ on horizontal lifts of an
orthonormal pair $ \widetilde{v}, \widetilde{w} \in T _{m,z} P$ of $v,w \in T _{z} ( \mathbb{C})$ is given
by 
\begin{equation} \label {eq.curvature}{\Omega} _{ \mathcal {A}} ( \widetilde{v},
\widetilde{w}) = -K _{\mathcal {A}}  (m, z),
\end{equation}
 for $m \in M _{z}$. For the
full construction the reader is referred to \cite[Section 6.4]{MS2}. The most
important property of $\Omega _{ \mathcal {A}}$ is that its integral over a section of $P$
asymptotic to $ \widetilde{\gamma} _{-},  \widetilde{\gamma} _{+}$, is
just 
\begin{equation} \label {eq.action} -(A _{H _{ +}} ( \widetilde{\gamma} 
_{+}) - A _{H _-} ( \widetilde{\gamma} _{-} )).
\end{equation}
There exists some $u \in \mathcal {M} ( \mathcal {A}, \widetilde{\gamma}
_{-}, \widetilde{\gamma}_+)$, for some $ \widetilde{\gamma} _{\pm}$ with $A
_{H_+} ( \widetilde{\gamma}_+)= \rho ( \widetilde{\phi}_+, [M])$ and $A _{H_-} ( \widetilde{\gamma}_-)= \rho (
\widetilde{\phi}_-, [M])$. And let $\overline {u} $ be some completely analogous
section corresponding to $ \overline{ \mathcal {A}}$, where $ \overline{
\mathcal {A}}$ is as in \eqref{eq.overline}. Then we have \begin{align}  \label
{eq.pair} 0 \leq \int _{u } \Omega _{ \mathcal {A} } + \int _{ u } K _{ \mathcal {A} } \pi ^{*} \omega _{st}, \\ 
 \label {eq.pair2}
 0 \leq \int _{ \overline{u} }   \Omega _{
\overline {\mathcal {A}} } + \int _{ \overline{u} } K _{
 \mathcal {\overline{A}} } \pi ^{*} \omega _{st}, 
\end{align} 
By \eqref{eq.action} we get
$$-\int _{u } \Omega _{ \mathcal {A} } = \rho (\phi _{+}, [M]) - \rho (\phi
_{-}, [M]),$$
where $\rho (\phi, [M])$ is the usual spectral invariant of Oh-Schwarz.
Similarly,  
$$-\int _{ \overline{u} } \Omega _{\overline{ \mathcal {A}} } = -(\rho
(\phi _{+} ^{-1}, [M]) - \rho (\phi _{-} ^{-1}, [M])).$$ Consequently,
subtracting \eqref{eq.pair}, \eqref{eq.pair2} we get
\begin{equation*} \rho (\phi_+, [M])+ \rho (\phi_+ ^{-1}, [M]) - \rho (\phi
_{-}, [M]) - \rho (\phi _{-} ^{-1}, [M])\leq L ^{+} _{j} ( \mathcal {A}) - L
^{-} _{j} ( { \mathcal {A}}),
\end{equation*} 
i.e. $$|\phi_+|_s -|\phi_-|_s \leq L^+
 _{j} ( \mathcal {A}) - L ^{-} _{j} ( \mathcal {A}).$$ 
\end{proof}

We will see in the next section that  for a suitable, stable Hofer
geodesic $p$, $ \mathcal {A} _{p}$ is quasi-flat. However,  the quasi-flat  
condition is much more subtle as we will show.
\begin{remark} The above sufficient condition for being  $L ^{s} $
(locally)-minimizing is likely necessary, but this is left as a question.
\end{remark}
 \section {Exotic extremals for $L ^{s}$ and Lalonde-McDuff Hamiltonian
 symplectomorphisms.}
 \label{section.exotic}

We will say that $ \{\mathcal {A} _{r}\}$, $0 \leq r \leq 1$ is a
\emph {\textbf{strong Cerf homotopy}} of  Hamiltonian connections on $M
\times S ^{1}$, in degree 2n, with holonomy of $ \mathcal {A} _0 = id$, and
holonomy of $ \mathcal {A}_1= \widetilde{\phi}$ if the following holds:
\begin{itemize} 
  \item $ \{ \mathcal {A} _{r}\}$ is constant for $r$ near $0$, $1$. 
  \item There are smooth sections $ u ^{i}$ of $M (\times [0,1] \times S ^{1})$
  whose restrictions over $r \times S ^{1} \subset [0,1] \times S ^{1}$ is $\gamma
  ^{i}_{r}$, where $\gamma ^{i} _{r}$ is a flat section of $ \mathcal {A} _{r}$. For some $D ^{i} _{0}$, $
\widetilde{\gamma} ^{i}_{0} = (\gamma ^{i}_{0}, D ^{i}_0)$ is a generator of $CF
_{2n} ( \mathcal {A} _{0})$, while
 $ \widetilde{\gamma} ^{i} _{1} = (\gamma ^{i}_{1}, D
_{1})$ is a generator of $CF _{2n} ( \mathcal {A} _{1})$, where $D _{1}$ is naturally
induced by $D _{0}$. Moreover we ask that  for some collection $ \{c_i\}$, $c_i
\cdot \widetilde{\gamma} ^{i}_1 $ represents $PSS ( [M]) \in FH _{2n} ( \mathcal
{A} _{1})$. 
\item The second condition also holds for the family $ \{ \overline{ \mathcal
{A}} _{r}\}$, defined analogously to $ \overline{ \mathcal {A}}$, with
associated sections $ \{ \overline{u} _{j}\}$.
\end {itemize}
\begin{proof} [Proof of Proposition \ref{thm.adiabatic}] Under assumptions of
the theorem, there is clearly an extension of $ \{ \mathcal {A} _{r}\}$ to a
morphism $ \mathcal {A} \in \mathcal {P} (id, \widetilde{\phi}_+)$, such that the sections
$u _{i}$ are $ \mathcal {A}$-flat.  We assume without loss of generality that $u
_{i}$  are regular, i.e. the associated CR operator is surjective for every $u
_{i}$, otherwise perturb $ \mathcal {A}$ so that this is satisfied, see
\cite[Chapter 8]{MS}.  Since $\int _{u_i} K _{
\mathcal {A}} \pi ^{*} \omega _{st} = A _{H_1} ( \widetilde{\gamma}) $, it follows that $L ^{+}
_{j} ( \mathcal {A})$ is at most 
$\rho ( \widetilde{\phi}, [M])$. By an identical argument  $L ^{-} _{j}
( \mathcal {A})$ is at most $- \rho ( \widetilde{\phi} ^{-1}, [M])$.
\end {proof}
\subsection {Lalonde-McDuff symplectomorphisms} \label{example.hofer1}
Let $H: S ^{2} \to \mathbb{R}$  generate a single rotation of
$S ^{2}, \omega$ in time $1+ \epsilon$. Define $H'$ to be a small time
dependent perturbation of $h \circ H$, for $h: \mathbb{R} \to \mathbb{R}$, with
$h''>0$ and sufficiently large so that the linearized time one flow at the
maximum/minimum of $H'$ is overtwisted, i.e. has non-constant periodic orbits
with period less than or equal to one. The time dependent perturbation is meant
to fix the maximum/minimum of $h \circ H$. The set of Hamiltonian symplectomorphisms formed by time one flow of such an $H'$, is essentially the entire set of
Lalonde-McDuff symplectomorphisms (they do not ask for a time dependent perturbation).
Set $ \mathcal {A} _{r}$ to be the connection
with holonomy path generated by $\eta (r) \cdot H'$, with $\eta$ as in
\eqref{eq.eta}.
 \begin{lemma} \label{lemma.cerf} The family $ \{
\mathcal {A} _{r}\}$ is a strong Cerf homotopy in degree 2n.
\end{lemma}
\begin{proof}   For $r $ less then some
threshold $r'$, there is a single generator $ \widetilde{\gamma} ^{r} _{\max}$ in
$ CF_2(\mathcal {A} _{r})$, corresponding to the the maximizer
$\max$ of $H'$. For $r$  just greater than $r'$ the linearized flow at $\max$ becomes
overtwisted, (and hence  $\widetilde{\gamma} ^{r}
_{\max}$ now has CZ index 4) and a new non-constant generator of CZ degree 2
appears in Floer complex of $ \mathcal {A} _{r}$, paired with a new generator of CZ index 3  so that at $r=r'_1$,
 there is bifurcation in the Floer chain complex. If $h''$ is sufficiently
 large, there will be similar bifurcations at $ \widetilde{\gamma} ^{r} _{\max}$
 for $r> r'_1$ creating more generators and more relations but non of this any longer happens for CZ degree 2. All of this is readily deduced simply from count of generators, (and we know them all
explicitely) the fact that CZ degree of $\widetilde{\gamma} ^{r} _{\max}$ goes up by 2 each time it gets more
overtwisted, and since the Floer homology is known. An identical
argument holds for the family $ \{ \overline{ \mathcal {A} _{}} _{r}\}$.
\end{proof}
\begin{proof} [Proof of Theorem \ref{thm.S2}] This follows by Proposition
\ref{thm.adiabatic}, and Lemma \ref{lemma.cerf} above.
\end{proof}
We end the section by noting that even for morphisms $ \mathcal {A} _{p}$,
coming from a quasi-autonomous path $p: [0,1] \to \text {Ham}(M, \omega)$,
locally minimizng the Hofer length functional, $L ^{s} ( \mathcal {A} _{p})$ may
not equal $L _{H} (p)$. 
 \begin{example} \label{example.quantum} The following very simple example was
 essentially suggested to me by Leonid Polterovich. 
Let $H: S ^{2}, \omega _{st} \to \mathbb{R}$, be a normalized Hamiltonian
generating Hamiltonian flow for a double rotation of $S ^{2}, \omega _{st}$ in
time $1+ \epsilon$, for a small $\epsilon$, and where $\int _{S ^{2}} \omega _{st}=1$.
Set $H'$ to be the result of a small time dependent perturbation, vanishing at
the maximazer $\max$ of $ h \circ H$, where $h'' \leq 0$, s.t the flow at the
poles becomes slow, i.e. so that linearized flow at the poles has no non-constant periodic orbits 
with period less than or equal to 1. Then by Ustilovsky \cite{U} time 1 flow of
$H'$ generates a stable Hofer geodesic, with $L ^{+}$ length close to 1. Set $
\widetilde{\phi}$ to be the time one map of $H'$, then $\rho (
\widetilde{\phi}, [M])$, is at most 1/2.
Consequently $PSS ( [M])$ is not represented by the constant orbit at $\max$.
It must then be represented by a new non-constant generator $ \widetilde{\gamma}
\in CF _{2} ( \widetilde{\phi})$. Now for $u \in \overline{ \mathcal {M}} (
\widetilde{\gamma})$ we have 
\begin{align*} \int _{u} K _{A} \pi ^{*} \omega _{st} = 
 \int _{u } \eta' (r)
 \cdot H _{p} \,\pi ^{*} \omega _{st}= \\ \int _{0} ^{1} \int _{-\infty} ^{\infty}
 \eta' (r)  H ( u (r, \theta), \theta) dr \wedge d \theta, 
\end{align*}
this integral is less than $L ^{+} (p)$, unless $u$ is a broken holomorphic
section, with principal component the constant section $u _{\max} (z) = \max$,
the maximizer of $H'$. But since we are allowing time dependent perturbations,
and $H'$ is assumed to be suitably generic, the 0 dimensional space $\overline{
\mathcal {M}} ( \widetilde{\gamma})$ will not have any broken elements of this
kind. It follows that $L ^{+} _{j} ( \mathcal {A} _{p}) \neq L ^{+} _{H} (p)$.
\end{example} 
\bibliographystyle{siam}  
\bibliography{link} 

\begin{thebibliography}{1}

\bibitem{GuileminSternberg}
{\sc V.~Guillemin, E.~Lerman, and S.~Sternberg}, {\em Symplectic fibrations and
  multiplicity diagrams}, Cambridge University Press, Cambridge, 1996.

\bibitem{Energy.stab.I}
{\sc F.~Lalonde and D.~McDuff}, {\em {Hofer's $L\sp \infty$-geometry: Energy
  and stability of Hamiltonian flows. I.}}, Invent. Math., 122 (1996),
  pp.~1--33.

\bibitem{LalondeFieldtheory}
{\sc F.~c. Lalonde}, {\em {A field theory for symplectic fibrations over
  surfaces.}}, Geom. Topol., 8 (2004), pp.~1189--1226.

\bibitem{MS2}
{\sc D.~McDuff and D.~Salamon}, {\em Introduction to symplectic topology},
  Oxford Math. Monographs, The Clarendon Oxford University Press, New York,
  second~ed., 1998.

\bibitem{MS}
\leavevmode\vrule height 2pt depth -1.6pt width 23pt, {\em $J$--holomorphic
  curves and symplectic topology}, no.~52 in American Math. Society Colloquium
  Publ., Amer. Math. Soc., 2004.

\bibitem{U}
{\sc I.~Ustilovsky}, {\em Conjugate points on geodesics of {H}ofer's metric},
  Differential Geom. Appl., 6 (1996), pp.~327--342.

\end{thebibliography}

\end{document}